\documentclass[english]{amsart}

\usepackage{babel}
\usepackage{amstext}
\usepackage{amsmath}
\usepackage{amscd}
\usepackage{amsfonts}
\usepackage{latexsym}
\usepackage{ifthen}
\newcommand{\lra}{\longrightarrow}

\newcommand\sF{{\mathcal F}}

\newcommand\sH{{\mathcal H}}

\newcommand\sD{{\mathcal D}}

\newcommand\sO{{\mathcal O}}

\newcommand\sC{{\mathcal C}}
\newcommand\sT{{\mathcal T}}

\newcommand\bZ{{\mathbb Z}}
\newcommand\bC{{\mathbb C}}

\newcommand\Hom{{\rm Hom}}
\newcommand\Ext{{\rm Ext}}

\newcounter{lemma}

\usepackage{enumerate}

\theoremstyle{plain} 
\newtheorem{theorem}{\noindent\bf Theorem}[section]
\newtheorem{lemma}[theorem]{\noindent\bf Lemma}

\newtheorem{proposition}[theorem]{\noindent\bf Proposition}
\newtheorem{claim}[theorem]{\noindent\bf Claim}
\newtheorem{definition}[theorem]{\noindent\bf Definition}

\theoremstyle{definition}
\newtheorem{remark}[theorem]{\noindent\bf Remark}

\title[saturation] 
{Bounded derived categories of very simple manifolds}
\author{Keiji Oguiso}

\dedicatory{Dedicated to Professor Dr. Fabrizio Catanese 
on the occasion of his sixtieth birthday}


\begin{document}

\begin{abstract} A non-representable cohomological functor of finite type of 
the bounded derived category of coherent sheaves of a compact complex manifold of dimension greater than one 
with no proper closed subvariety is given explicitly in categorical terms. This is a partial generalization of an impressive result due to Bondal and 
Van den Bergh.
\end{abstract}
\maketitle
\tableofcontents
\section{Introduction - Background and Main Result}
\noindent

Throughout this paper, we shall work over the complex number field 
$\bC$. 

Let $T$ be a $\bC$-linear $\Ext$-finite triangulated category. That is, $T$ is a triangulated category (\cite{BBD83}, Chapter 1 (1.3), 
\cite{KS90}, Chapter I (1.5)) such that 
${\rm Hom}_{T}(a, b)$ for any $a, b \in {\rm Ob}\, (T)$ is a $\bC$-linear 
space satisfying 
$$\sum_{n \in \bZ} \dim_{\bC} \Hom_T(a, b[n]) 
< \infty\, .$$ 
For instance, the bounded derived category $D^b({\rm Coh}\,X)$ of coherent sheaves on a compact connected complex manifold naturally forms a $\bC$-linear $\Ext$-finite 
triangulated category (\cite{GR84} Annex, for basics on the coherent sheaves in the analytic setting). 

By a {\it cohomological functor} on $T$, 
we mean a contravariant functor 
$$H : T \longrightarrow ({\rm Vect-}{\bC})$$
such that for any distinguished triangle 
$$a \longrightarrow b \longrightarrow c \longrightarrow a[1]$$
in $T$, the induced sequence
$$H(a[1]) \longrightarrow H(c) \longrightarrow H(b) \longrightarrow H(a)$$
is an exact sequence of $\bC$-linear spaces. $H$ is {\it of finite type}, 
if in addition,     
$$\sum_{n \in \bZ} \dim_{\bC} H(b[n]) 
< \infty$$
for any $b \in {\rm Ob}\, (T)$. If $T$ is a $\bC$-linear $\Ext$-finite triangulated category, then the functors $\Hom_T(-,a)$ ($a \in {\rm Ob}\, (T))$ 
are cohomological functors of finite 
type. A functor $H$ is called {\it representable} if there is 
$a \in\, {\rm Ob}\,(T)$ such that 
$H(-) \simeq {\rm Hom}_T(-, a)$ as functors. $T$ is called {\it saturated} if all cohomological functors of finite type on $T$ are representable. 

In their paper \cite{BV03}, Corollary 3.1.5, Theorem 5.6.3, Bondal and 
Van den Bergh show the following remarkable 
theorems:

\begin{theorem}\label{bb1}  Let $X$ be a smooth, proper 
algebraic variety. Then, $D^b({\rm Coh}\, X)$ is saturated.
\end{theorem}

\begin{theorem}\label{bb2}  Let $X$ be a smooth, compact complex surface 
having no complete curve. Then, $D^b({\rm Coh}\, X)$ is not saturated.
\end{theorem}

Smooth proper algebraic varieties 
and compact complex manifolds share many properties 
in coherent sheaf cohomology at fundamental levels. So, it is rather surprizing at least for me that in the bounded derived categories of coherent sheaves, they make such a sharp contrast already in dimension $2$. 

The aim of this note is to generalize Theorem (\ref{bb2}) a little bit. 

We call a compact connected complex manifold $X$ {\it very simple} if 
$\dim_{\bC}X \ge 2$ and $X$ has no complete irreducible subvariety other than $X$ itself and a point. The notion ``very simple" is much more restrictive than the notion 
``simple" introduced by Fujiki (\cite{Fu83} Page 237, line 7). Any very simple manifold has no global meromorphic function so that it is far from being algebraic. As well-known, generic complex tori of dimension $\ge 2$, K3 surfaces of Picard number $0$, some surfaces belonging to the Kodaira's class $VII$ are very simple manifolds (\cite{Ve04} Proposition (3.1) and \cite{BHPV04}, Page 229, Proposition (19.1)).

Our main result is the following: 

\begin{theorem}\label{main}  Let $X$ be a very simple manifold. 
Then, $D^b({\rm Coh}\,X)$ is not saturated. 
\end{theorem}

We also give a non-representable cohomological functor of finite type 
explicitly in categorical terms (Theorem(\ref{unrepresented})). 

As in Theorem (\ref{bb2}), one of crucial observations 
is that the heart $^{p}{\rm Coh}\, X$ of the $t$-structure given by tilting by 
torsion pair is of finite length (Theorem (\ref{artinnoether})). {\it As the referee pointed out, Theorem (\ref{artinnoether}) is also a special case of a result of Meinhardt} (\cite{Me07}, Proposition 3.5, the case $p =1$) when $X$ is a very simple complex tori, and the argument there is in fact valid for any very simple manifold. In their proof of Theorem (\ref{bb2}), Bondal and Van den Bergh showed this property by using the invariance $^{p}{\rm Coh}\, X$ under the derived dual, which is true only in dimension $2$. As in \cite{Me07}, we replace this argument just by a simple diagram chasing. As it is so simple, we shall give a full proof, too. Bondal and Van den Bergh then derived unsaturatedness in Theorem (\ref{bb2}) {\it by argue by contradiction} based on the uniqueness of Serre functor and the 
classification 
of saturated derived categories of finite length (\cite{RV02}, Lemma V.1.1). This is 
an elegant argument but tell nothing about which functor is in fact non-representable. We replace this argument again by a more elementary one which is based on the embedding of simple objects into 
their injective hulls in an enlarged category. This provides us an explicit example of non-representable cohomological functor of finite type in categorical 
terms. The argument here is inspired by a very impressive example in \cite{BV03} Section 2, 2.5 and 
an argument in \cite{RV02} Lemma V.1.1.

Besides \cite{BV03}, \cite{Me07} and \cite{RV02}, our work is also related to \cite{Ne96}, 
\cite{Ro08}, \cite{TV08}, \cite{Ve04} and \cite{Ve08} in some sense. 
\par
\vskip 4pt
\noindent {\it Acknowledgements.} This note grew out from the working seminar 
on derived categories and DG-categories held at Osaka University in Fall 2009 
under Professor Atsushi Takahashi. First of all, I would like to express 
my best thank to him for showing me the very interesting paper \cite{BV03} and for asking me to introduce this paper in his seminar. The members there, especially Professors Atsushi Takahashi, Akira Fujiki, Moto-o Uchida and Mister Kotaro Kawatani, gave me several valuable suggestions, corrections and comments. I would like to express my thank to all of the seminar members. I also would like to thank to the referee for informing an important paper \cite{Me07} and for pointing out many mistakes in the first and the second versions. Last but not least at all, it is my honor to dedicate this note to Professor Dr. Fabrizio Catanese on the occasion of his sixtieth birthday, from whom I learned much about mathematics and other things since I met him in 1993 in Bonn.

\section{Proof of Main Theorem}
\noindent

Let $X$ be a very simple manifold. From now, for simplicity, 
we denote the abelian category ${\rm Coh}\, X$ 
of coherent sheaves on $X$ by $\sC$ 
and the bounded derived category $D^b({\rm Coh}\, X) = D^b(\sC)$ of coherent sheaves on $X$ by $\sD$. We note that 
$\sC$ is then a full subcategory of $\sD$. 
In particular, $\Hom_{\sC}(a, b) = \Hom_{\sD}(a, b)$ 
for any $a, b \in {\rm Ob}\,(\sC)$. 

Let $\sT$ be the full subcategry of $\sC$ consisting of coherent 
torsion sheaves on $X$. We denote by $\sF$ 
the full subcategory of $\sC$ consisting of coherent torsion free sheaves 
on $X$ (\cite{GR84}, Annex for definitions in analytic setting). 

\begin{lemma}\label{torsion} Consider the ordered pair $(\sT, \sF)$ of 
$\sT$ and $\sF$ defined above. 
\begin{enumerate} 
\item $(\sT, \sF)$ is a torsion pair of $\sC$ in the sense that this satisfies 

(i) $\Hom_{\sC}(t, f) = 0$ for any $t \in {\rm Ob}\, (\sT)$ 
and $f \in {\rm Ob}\, (\sF)$, and 

(ii) For any $c \in {\rm Ob}\, (\sC)$, there are 
$t \in {\rm Ob}\, (\sT)$, $f \in {\rm Ob}\, (\sF)$ 
and an exact sequence in $\sC$:
$$0 \lra t \lra c \lra f \lra 0\,\, .$$
\item The torsion pair $(\sT, \sF)$ is cotilting of $\sC$ in the sense that 
for any $c \in {\rm Ob}\, (\sC)$, there are $h \in {\rm Ob}\, (\sF)$
and an epi-morphism $h \lra c$ in $\sC$. 
\end{enumerate} 
\end{lemma}
\begin{proof}
The assertion (1) is clear. Let us show the assertion (2). The statement is clearly true 
if $\sC$ has enough projectives. However, this is known only in 
$\dim\, X \le 2$ (\cite{Sch82} Theorem 2). So, we have to take another approach. The approach here is identical to \cite{BV03} Section 5, 5.6. Steps 2 and 3 for surfaces. {\it We repeat their argument here just to make sure that our assumption ``very simple" is enough to conclude.} Take $c \in {\rm Ob}\, (\sC)$ and consider the exact sequence
$$0 \lra t \lra c \lra f \lra 0$$
which exists due 
to (ii), where $t \in {\rm Ob}\, (\sT)$ and $f \in {\rm Ob}\, (\sF)$. 
Recall that $\sC = {\rm Coh}\, X$ is a noetherian category for any compact complex manifold $X$, that is, for a given coherent sheaf, any ascending chain 
of its subsheaves is stationally. For instance, this follows from the induction on the dimension of the support. 
Take then (one of) the maximal subsheaf $d \subset c$ such that 
$d \cap t = 0$. Since $\sC$ is noetherian and $0 \cap t = 0$, such a subsheaf $d$ certainly exists (by one of the three equivalent definitions 
of noetherian property). Since $d \subset c$, $d \cap t = 0$ 
and $t$ is the torsion part of $c$, we have $d \in \sF$. Put $s := c/d$. 
Then, by the choice of $d$, it follows that 
$t \subset s$ and that $s$ is an esstential extension of $t$, in the sense that $s' \cap t \not= 0$ for any non-zero subsheaf $s'$ of $s$. Now we use the fact that 
$X$ is very simple. Then ${\rm Supp}\, (t)$ consists of finitely many points, 
say, $\{x_1, x_2, \cdots , x_n \}$. Thus, there are very large integer 
$N$ such that 
$$t \cap (\prod_{i=1}^{n} {\mathbf m}_{X, x_i})^N s = 0\,\, ,$$
where ${\mathbf m}_{X, x_i}$ is the maximal ideal sheaf of $x_i$. 
(This is true for any coherent sheaf containing $t$.) Since $s$ 
is an esstential extension of $t$, it follows that 
$$(\prod_{i=1}^{n} {\mathbf m}_{X, x_i})^N s = 0\,\, ,$$
whence ${\rm Supp}\, (s) = {\rm Supp}\, (t)$. In particular, 
${\rm Supp}\, (s)$ 
also consists of finitely many points. Hence, there are very large integer 
$M$ and an epi-morphism 
$\varphi : \sO_{X}^{\oplus M} \lra s$ in $\sC$. Then, pulling back the exact sequence
$$0 \lra d \lra c \lra s = c/d \lra 0$$
by $\varphi$, we have a commutative diagram of exact lows:
$$
\begin{CD} 
0 @>>> d @>>> h @>>> \sO_X^{\oplus N} @>>> 0\\ 
@. @VV{id}V @VV{\tilde{\varphi}}V @VV{\varphi}V @.\\ 
0 @>>> d @>>> c @>>> s @>>> 0
\end{CD}
$$
By the five lemma, $\tilde{\varphi}$ is also an epi-morphism. Recall also that 
$d \in \sF$. Thus $h \in {\rm Ob}\, (\sF)$, because $\sF$ is closed under extensions. This proves the assertion (2).
\end{proof}
\begin{definition} We define the full subcategories 
$^p\sD^{\ge 0}$, $^p\sD^{\le 0}$ and $^p\sC$ of $\sD$ by
$$^p\sD^{\ge 0} := \{x \in {\rm Ob}\,(\sD)\, \vert\, \sH^0(x) \in 
{\rm Ob}\,(\sF), \sH^{i}(x) = 0 \forall i \le -1\}$$
$$^p\sD^{\le 0} := \{x \in {\rm Ob}\,(\sD)\, \vert\, \sH^1(x) \in 
{\rm Ob}\,(\sT), \sH^{i}(x) = 0 \forall i \ge 2\}$$
$$^p\sC\, =\, ^p\sD^{\ge 0} \cap ^p\sD^{\le 0}\, =\,  
\{x \in {\rm Ob}\,(\sD)\, \vert\, \sH^0(x) \in 
{\rm Ob}\,(\sF), \sH^1(x) \in 
{\rm Ob}\,(\sT), \sH^{i}(x) = 0 \forall i \not= 0, 1\}\,\, .$$
Here $\sH^i(x) \in {\rm Ob}\, (\sC)$ is the $i$th cohomology sheaf 
of $x \in {\rm Ob}\,(\sD)$. 
\end{definition}
By definition, both $\sF$ and $\sT[-1]$ are full subcategories 
of $^p\sC$. 
\begin{theorem}\label{tstructure} Under the notation above, 

\begin{enumerate} 
\item The full subcategory $\sC$ of $\sD$ is an abelian category, in which 
$$0 \lra a \xrightarrow{\alpha} b 
\xrightarrow{\beta} c \lra 0$$
is exact if and only if 
$$a \xrightarrow{\alpha} b \xrightarrow{\beta} c \lra a[1]$$
is a distinguished triangle in $\sD$ and $a, b, c \in\, {\rm Ob}\, (\sC)$. 
\item The full subcategory $^p\sC$ is an abelian category, in which 
$$0 \lra a \xrightarrow{\alpha} b 
\xrightarrow{\beta} c \lra 0$$
is exact if and only if 
$$a \xrightarrow{\alpha} b \xrightarrow{\beta} c \lra a[1]$$
is a distinguished triangle in $\sD$ and $a, b, c \in\, {\rm Ob}\, (^p\sC)$. 
\item The original $\sD$ is also the bounded derived category 
of the new abelian category $^p\sC$, i.e., $\sD = D^b(^p\sC)$. Moreover, 
the ordered pair $(\sF, \sT[-1])$ forms a torsion 
pair of $^p\sC$ in the sense of Lemma (\ref{torsion}) (1) and 
it is now tilting of $^p\sC$ in the sense that for any $x \in\, {\rm Ob}\, (^p\sC)$, there are $f \in {\rm Ob}\, (\sF)$ and a mono-morphism $x \lra f$ 
in $^p\sC$. 
\end{enumerate} 
\end{theorem}

\begin{proof} 
Since $\sC$ is the heart of $\sD$ under the standard $t$-structure 
of $\sD = D^b(\sC)$ (\cite{KS90} Examples 10.1.3, (i)), the assertion (1) follows 
from \cite{KS90} Proposition 10.1.11 and its proof. 
By \cite{HRS96} Chapter I, Proposition 2.1 (with shift by $1$), the pair of full subcategories $^p\sD^{\ge 0}$ and $^p\sD^{\le 0}$ defines another $t$-structure of $\sD$. Hence the new heart $^p\sC$ is an abelian category with desired characterization 
of exact sequences by \cite{KS90} Proposition 10.1.11 and its proof. The second half is proved 
by \cite{HRS96} Chapter I, Proposition 3.2.  Recall also that the torsion pair 
$(\sT, \sF)$ 
of $\sD$ is cotilting. Thus, the first part of the statement (3) follows from 
\cite{BV03} Proposition 5.4.3, which is based on the main result of 
\cite{Ne90}. 
\end{proof}
\begin{remark} In $\sC$, we have a standard exact sequence 
$$0 \lra {\bf m}_{X, x} \xrightarrow{\iota} \sO_X \lra \bC_x \lra 0$$
where $x$ is a closed point of $X$ and ${\bf m}_{X, x}$ 
is the maximal ideal sheaf of $x \in X$. This gives a distinguished 
triangle
$${\bf m}_{X, x} \xrightarrow{\iota} \sO_X \lra \bC_x \lra 
{\bf m}_{X, x}[1]\, ,$$
and then another distinguished triangle 
$$\bC_x[-1] \lra {\bf m}_{X, x} \xrightarrow{\iota} \sO_X \lra \bC_x = 
(\bC_x[-1])[1]\,$$
as well. In the second triangle, we have 
$\bC_x[-1] \in {\rm Ob}\, (\sT[-1])$ and ${\bf m}_{X, x} 
\in {\rm Ob}\, (\sF)$ and $\sO_X \in {\rm Ob}\, (\sF)$. 
Thus, we obtain an exact sequence
$$0 \lra \bC_x[-1] \lra {\bf m}_{X, x} \xrightarrow{\iota} \sO_X \lra 0$$
in $^p\sC$. In particular, the natural 
inclusion $\iota : {\bf m}_{X, x} \lra \sO_X$ in $\sC$ is no longer a mono-morphism in the new 
abelian category $^p\sC$. 
For essentially the same reason, the strictly descending chain in $\sC$
$$\sO_X \supset {\bf m}_{X, x} \supset {\bf m}_{X, x}^2 \supset \ldots \supset 
 {\bf m}_{X, x}^n \supset \ldots$$
is no longer a descending chain in $^p\sC$. 

However, we should also note that if 
$X$ contains a positive dimensional proper subvariety, say $Y \not= X$, 
then $\sO_Y$ and ${\bf m}_{Y, x}^n$ ($n$ being any 
positive integer) are elements of ${\rm Ob}\, (T)$ and
$$\sO_Y \supset {\bf m}_{Y, x} \supset {\bf m}_{Y, x}^2 
\supset \ldots \supset 
 {\bf m}_{Y, x}^n \supset \ldots$$
is a strictly descending chain in $\sC$ and 
$$\sO_Y[-1] \supset {\bf m}_{Y, x}[-1] \supset {\bf m}_{Y, x}^2[-1] 
\supset \ldots \supset 
 {\bf m}_{Y, x}^n[-1] \supset \ldots$$
is also a strictly descending chain in $^p\sC$.

So, in the next theorem, the fact that $X$ is very simple 
is crucial.
\end{remark}

\begin{theorem}\label{artinnoether} 
$^p\sC$ is of finite length, i.e., both noetherian and 
artinian.
\end{theorem}
As remarked in the introduction, this is also a special case 
of \cite{Me07}, Proposition 3.5. However, as 
the fact is very crucial for us and the proof is so simple, 
we shall give a full proof.
\begin{proof} By definition, it suffices to show the following two:
\begin{proposition}\label{artin}  $^p\sC$ is artinian.
\end{proposition}
\begin{proposition}\label{noether} $^p\sC$ is noetherian.
\end{proposition}

\begin{proof} Here we shall show Proposition (\ref{artin}). Let 
$a = a_{0} \in {\rm Ob}\,(^p\sC)$ 
and 
$$a_0 \supset a_1 \supset \ldots \supset a_n \supset \ldots$$
be a descending chain of subobjects of $a$ in $^p\sC$. Recall that 
$(\sF, \sT[-1])$ 
forms a torsion pair of $^p\sC$. Thus, for each integer $n \ge 0$, there are $f_n \in {\rm Ob}\, (\sF)$, $t_n \in {\rm Ob}\, (\sT)$ 
and an exact sequence in $^p\sC$:
$$0 \lra f_n \lra a_n \lra t_n[-1] \lra 0\, .$$ 
As $(\sF, \sT[-1])$ is a torsion pair of $^p\sC$, it follows that $\Hom_{^p\sC}(f_n, t_m[-1]) = 0$. In particular, for each 
$$n > m\, ,$$
via the monomorphism $a_m \supset a_n$, we have 
$f_m \supset f_n$ and hence, a sequence of mono-morphisms 
$$f_0 \supset f_1 \supset \ldots \supset f_n \supset \ldots$$ in $^p\sC$. Let 
$g_{n, m} = f_m/f_n$ in $^p\sC$. 
Since $g_{n, m} \in {\rm Ob}\, (^p\sC)$ and $(\sF, \sT[-1])$ 
is a torsion pair of $^p\sC$, there are $h_{n,m} \in {\rm Ob}\, (\sF)$, 
$s_{n,m} \in {\rm Ob}\, (\sT)$ and an exact sequence
$$0 \lra h_{n,m} \lra g_{n,m} \lra s_{n,m}[-1] \lra 0$$
in $^p\sC$. Thus, we obtain a commutative diagram of exact sequences
$$
\begin{CD}
@. @. @. 0       @.\\
@. @. @. @VVV @.\\ 
@. @. @. h_{n,m} @.\\
@. @. @. @VVV @.\\ 
0 @>>> f_n @>>> f_m @>>{\tau}> g_{n,m} @>>> 0\\
@. @. @. @VV{\pi}V @.\\  
@. @. @. s_{n,m}[-1] @.\\
@. @. @. @VVV @.\\ 
@. @. @. 0 @.\\
\end{CD}
$$
Since $(\sF, \sT[-1])$ is a torsion pair of $^p\sC$, it follows that $\Hom(f_m, s_{n, m}[-1]) = 0$. Thus $\pi \circ \tau = 0$. On the other hand, since $\tau$ and $\pi$ are epi-morphisms in $^p\sC$, so is $\pi \circ \tau$ in $^p\sC$. 
Thus $s_{n,m}[-1] \simeq 0$. Hence $g_{n, m} \simeq h_{n,m}$ in $^p\sC$, 
whence we have an exact sequence 
$$0 \lra f_n \lra f_m \lra h_{n,m} \lra 0$$
in $^p\sC$, whence a distinguished triangle
$$f_n \lra f_m \lra h_{n,m} \lra f_{n}[1]$$ 
in $\sD$. Here $f_n$, $f_{m}$, $h_{n,m}$ are in ${\rm Ob}\, (\sF)\, 
\subset\, {\rm Ob}\, (\sC)$. Thus, we have an exact sequence
$$0 \lra f_n \lra f_m \lra h_{n,m} \lra 0$$
also in $\sC = {\rm Coh}\,X$. Thus $h_{n,m} = f_m/f_n$ 
also in $\sC$. For an object $\sC = {\rm Coh}\, X$, we can speak 
of its rank and we obtain the key inequality (I):
$$\infty > {\rm rank}\,(f_0) \ge {\rm rank}\,(f_1) \ge \ldots \ge {\rm rank}\,(f_n) \ge \ldots\, .$$
Thus ${\rm rank}\,(f_n)$ is constant for all large $n$, say for all $n \ge N$. 
Since $h_{n,m} = f_m/f_n$ is tosion free, 
it follows that $f_n$ are isomorphic for all $n \ge N$ under the same morphisms in the exact sequence any one of above. Thus $f_n$ are all the same for 
all $n \ge N$ as subobjects of $a_0$ in $^p\sC$. By 
$t_n[-1] \simeq a_n/f_n = a_n/f_N$ in $^p\sC$, it follows that the natural inclusion $a_n \subset a_m$ induces a descending chain of inclusions 
$$t_N[-1] \supset t_{N+1}[-1] \supset \ldots \supset t_{m}[-1] \supset \ldots \,\, .$$
Let $m \ge N$ and set $q_{n,m} = t_{m}[-1]/t_{n}[-1]$ in $^p\sC$. 
Then, as before, we obtain a commutative diagram of exact sequences in $^p\sC$:
$$
\begin{CD}
@. @. @. 0       @.\\
@. @. @. @VVV @.\\ 
@. @. @. \ell_{n,m} @.\\
@. @. @. @VV{\beta_{n,m}}V @.\\ 
0 @>>> t_n[-1] @>>{\iota_{n,m}}> t_m[-1] @>>{\alpha_{n,m}}> q_{n,m} @>>> 0\\
@. @. @. @VV{\gamma_{n,m}}V @.\\  
@. @. @. r_{n,m}[-1] @.\\
@. @. @. @VVV @.\\ 
@. @. @. 0 @.\\
\end{CD}
$$
where $\ell_{n,m} \in {\rm Ob}\, (\sF)$ and $r_{n,m} \in {\rm Ob}\, (\sT)$. 
Put $\delta_{n,m} := \gamma_{n,m} \circ \alpha_{n,m}$. As before, 
$\delta_{n,m}$ is an epi-morphism in $^p\sC$. 
Then, by setting $K_{n, m} = {\rm Ker}\, (\delta_{n,m})$, 
we have an exact sequence in $^p\sC$:
$$0 \lra K_{n,m} \xrightarrow{\kappa_{n,m}} t_m[-1] 
\xrightarrow{\delta_{n,m}} r_{n,m}[-1]  \lra 0\, .$$
Since $(\sF, \sT[-1])$ is a torsion pair of $^p\sC$, there are $x \in {\rm Ob}\,(\sF)$, $k_{n,m} \in {\rm Ob}\, (\sT)$ and an exact sequence
$$0 \lra x  \xrightarrow{o} K_{n,m} \lra k_{n,m}[-1] \lra 0\, .$$
Since $o$ and $\kappa_{n,m}$ are mono-morphisms in $^p\sC$, so is 
$\kappa \circ o$. On the other hand, since $(\sF, \sT[-1])$ is a torsion 
pair of $^p\sC$, $x \in {\rm Ob}\, (\sF)$ and $k_{n,m}[-1] \in {\rm Ob}\, (\sT[-1])$, it follows that the mono-morphism $\kappa \circ o$ is at the same time 
$0$. Thus, $x = 0$ in $^p\sC$, 
whence 
$$K_{n,m} = k_{n,m}[-1] \in {\rm Ob}\, (\sT[-1])\, .$$
By 
$$\gamma_{n,m} \circ \alpha_{n,m} \circ \kappa_{n,m} = \delta_{n,m} \circ \kappa_{,m} = 0\, ,$$
we have a morphism $p_{n,m} : k_{n,m}[-1] \lra \ell_{n,m}$ and we have 
the following commutative diagram in $^p\sC$ with exact lows:
$$
\begin{CD}
0 @>>> k_{n,m}[-1] @>>{\kappa_{n,m}}> t_m[-1] @>>{\delta_{n,m}}> 
r_{n,m}[-1] @>>> 0\\
@. @V{p_{n,m}}VV @V{\alpha_{n,m}}VV @V{id}VV @.\\ 
0 @>>> \ell_{n,m} @>{\beta_{n,m}}>> q_{n,m} @>>{\gamma_{n,m}}> r_{n,m}[-1] 
@>>> 0\\
\end{CD}
$$ 
Then, by the snake lemma, we have then exact sequences in $^p\sC$:
$$0 \lra {\rm Ker}\,(p_{n,m}) \lra {\rm Ker}\,(\alpha_{n,m}) = t_n[-1] \lra 
{\rm Ker}\,(id) = 0$$
$${\rm Ker}\,(id) = 0 \lra {\rm Coker}\,(p_{n,m}) \lra 
{\rm Coker}\,(\alpha_{n,m}) = 0\, .$$
Thus, ${\rm Ker}\,(p_{n,m}) = t_n[-1]$ and ${\rm Coker}\,(p_{n,m}) = 0$. 
Hence, we obtain an exact sequence in $^p\sC$:
$$0 \lra t_n[-1] \xrightarrow{\iota_{n,m}'} k_{n, m}[-1] 
\xrightarrow{p_{n,m}} \ell_{n,m} \lra 0\, .$$
Then, the sequence 
$$t_n[-1] \xrightarrow{\iota_{n,m}'} k_{n, m}[-1] 
\xrightarrow{p_{n,m}} \ell_{n,m} \lra t_n = (t_n[-1])[1]\,$$
is a distinguished triangle in $\sD$. From this, we also obtain distinguished 
traiangles
$$k_{n,m}[-1] \lra \ell_{n,m} \lra t_n \lra k_{n,m}\, ,$$
$$\ell_{n,m} \lra t_n \lra k_{n,m} \lra \ell_{n,m}[1]$$ 
in $\sD$. Since $\ell_{n,m}$, $t_n$ 
and $k_{n,m}$ are in $\sC$, we obtain an exact sequence in $\sC$:
$$0 \lra \ell_{n,m} \lra t_n \lra k_{n,m} \lra 0\, .$$
Since $\ell_{n,m} \in \sF$ but $t_n \in \sT$, it follows that 
$\ell_{n,m} \simeq 0$ in $\sC$, and therefore also in $\sF$. Thus $\ell_{n,m} \simeq 0$ 
in $^p\sC$ as well. Hence $q_{n,m} \simeq r_{n,m}[-1]$ in $^p\sC$, and we 
have an exact sequence in $^p\sC$:
$$0 \lra t_n[-1] \xrightarrow{\iota_{n,m}} t_m[-1] 
\xrightarrow{\delta_{n,m}} r_{n,m}[-1] \lra 0\, .$$
We then obtain disinguished triangles in $\sD$:
$$t_n[-1] \xrightarrow{\iota_{n,m}} t_m[-1] 
\xrightarrow{\delta_{n,m}} r_{n,m}[-1] \lra t_n\, ,$$
$$t_n \xrightarrow{\iota_{n,m}[1]} t_m 
\xrightarrow{\delta_{n,m}[1]} r_{n,m} \lra t_n[1]\, ,$$
and therefore an exact sequence in $\sC$:
$$0 \lra t_n \xrightarrow{\iota_{n,m}[1]} t_m \lra r_{n,m} \lra 0\, .$$
Recall that our manifold $X$ is very simple. Then, for any torsion sheaf $t$, 
the support of $t$ consists of finitely many points and 
one can speak of the (finite) length $\ell(t_n)$. 
One has then the key inequality (II):
$$\ell(t_N) \ge \ell(t_{N+1}) \ge \ldots \ge \ell(t_m) \ge \ldots \, .$$
Thus, $\ell(t_m)$ are stationally for all large $m$, say for all $m \ge M$ with $M \ge N$. 
Hence, under the mono-morphism $\iota_{n,m}[1]$, we have $t_m = t_n$, for all 
$n > m \ge M$, in $\sC$, whence also in $\sD$. Thus, under the morphism 
$\iota_{n,m}$, 
we have that $t_m[-1] = t_n[-1]$ for all $n > m \ge M$ in $\sD$, 
whence, in $\sT[-1]$ as well. Thus, $a_n = a_m$, for all $n > m \ge M$, in 
$\sD$. This means that $^p\sC$ is artinian. 
\end{proof}

\begin{proof} Proposition (\ref{noether}) is proved by 
\cite{BV03} Section 5, 5.6. Step 4 based on Lemma 5.5.2 in \cite{BV03}. Proposition (\ref{noether}) also follows from a few minor modifications of the proof of Proposition (\ref{artin}) as follows. 

Let $a = a_{\infty} \in\, {\rm Ob}\, (^p\sC)$ and 
$$a_1 \subset a_2 \subset \ldots \subset a_n \subset \ldots \subset 
a_{\infty}$$
be an ascending chain of subobjects of $a_{\infty}$ in $^p\sC$. As in the proof of Proposition 
(\ref{artin}), for each $1 \le n \le \infty$, there are $f_n \in {\rm Ob}\, (\sF)$, $t_n \in {\rm Ob}\, (\sT)$ 
and an exact sequence
$$0 \lra f_n \lra a_n \lra t_n[-1] \lra 0$$
in $^p\sC$. Then for $1 \le n, m \le \infty$ {\it with inequality reversed}
$$n < m\, $$
we obtain exactly the same commutative diagrams, exact sequences and distinguished triangles as in the proof of Proposition (\ref{artin}). 
Thus, corresponding to the key inequalities (I), (II) in the proof of Proposition (\ref{artin}), we have
$${\rm rank}\, (f_1) \le\, {\rm rank}\,(f_2) \le \ldots \le\, 
{\rm rank}\,(f_n) 
\le \ldots \le {\rm rank}\,(f_{\infty}) < \infty\, ,$$
$$\ell(t_{N}) \le \ell(t_{N+1}) \le \ldots \le \ell(t_{m}) \le \ldots \le 
\ell(t_{\infty}) < \infty\, .$$
From these two inequalities together with the same commutative diagrams, exact sequences and distinguished triangles in the proof of Proposition (\ref{artin}), we can find a positive integer $M$ such that $a_m = a_M$ for all $m \ge M$ 
exactly as the same way as in the proof of Proposition (\ref{artin}). 
\end{proof}

This completes the proof of Theorem (\ref{artinnoether}). 
\end{proof}

By Theorem (\ref{artinnoether}), any object of $^p\sC$ is given by a finite successive extensions of simple objects of $^p\sC$. Let $\{c_{\lambda}\}_{\lambda \in \Lambda}$ be a set of complete representatives of isomorphism classes of simple objects of $^p\sC$. 

Let $^p\sC^{\rm op}$ be the opposite category of $^p\sC$. Then, since $^p\sC$ is noetherian as observed above, it follows that $^p\sC^{\rm op}$ is artinian. 
Let $\hat{^p\sC}$ be the full subcategory of the category of functors 
${\rm Funct}(^p\sC^{\rm op}, ({\rm Vect-}{\bC}))$ consisting of left exact functors. ($\hat{^p\sC}$ is denoted by ${\rm Sex}\, (^p\sC^{\rm op}, ({\rm Vect-}{\bC}))$ in \cite{Ga62}, but we will not use this notation.) Then, we have a natural {\it covariant} functor
$$\iota : (^p\sC) \lra \hat{^p\sC}\, ;\, x \mapsto \Hom_{^p\sC}(-, x)\, .$$
Here $^p\sC = (^p\sC)$. By \cite{Ga62}, Page 354, Proposition 6, $\iota$ is an exact functor and by \cite{Ga62}, Page 356, Th\'eor\`eme 1, $\hat{^p\sC}$ is locally noetherian 
(see \cite{Ga62}, Page 356, for definition) and $\iota$ is fully faithful. 
Moreover, by \cite{Ga62}, Page 362, Th\'eor\`eme 2, it follows that 
one can find an injective hull (called enveloppe injective in \cite{Ga62}) of each object of $\hat{^p\sC}$ {\it inside} $\hat{^p\sC}$. 
For each $c_{\lambda}$, let us choose and fix an injective hull 
$\hat{c_{\lambda}} \in {\rm Ob}\, (\hat{^p\sC})$ of $\iota(c_{\lambda})$ 
and set 
$$e = \oplus_{\lambda \in \Lambda} \hat{c_{\lambda}}\,\,.$$
This $e$ is an object of $\hat{^p\sC}$, because $\hat{^p\sC}$ is locally 
noetherian, and is in fact an injective object by \cite{Ga62}, Page 358, 
Corollaire 1. Since $\iota$ is an exact functor, $\iota$ gives rise to a functor 
between the bounded derived categories:
$$\tilde{\iota} : D^b(^p\sC) \lra D^b(\hat{^p\sC})\, .$$

Now the following theorem completes the proof of our main theorem:

\begin{theorem}\label{unrepresented} 
Let $E(-) = \Hom_{D^b(\hat{^p\sC})}(\tilde{\iota}(-), e)$. 
\begin{enumerate}
\item The functor $E(-)$ is a cohomological functor of finite type on $\sD = D^b(\sC) \simeq D^b( ^p\sC)$.
\item The functor $E(-)$ is not representable 
on $\sD$.
\end{enumerate}
\end{theorem}

\begin{proof} It is clear that $E(-)$ is a cohomological functor. Let 
us show that $E(-)$ is of finite type. Let $c \in {\rm Ob}\, (\sD) = {\rm Ob}\, (D^b(^p\sC))$. 
Since $e$ is an injective object in $\hat{^p\sC}$, it follows that 
$$E(c) = \Hom_{D^b(\hat{^p\sC})}(\tilde{\iota}(c), e) = 
\Hom_{K^b(\hat{^p\sC})}(\tilde{\iota}(c), e)\, .$$ 
Here $K^b(\hat{^p\sC})$ is the homotopy category of bounded complexes of 
$\hat{^p\sC}$. Since $e$ is only in degree $0$ in $K^b(\hat{^p\sC})$ and $\tilde{\iota}(c)$ is represented by a bounded complex of objects of $^p\sC$, 
it is then sufficient to show that 
$$\dim_{\bC}\, \Hom_{\hat{^p\sC}}(\iota(c), e) < \infty$$
for all $c \in {\rm Ob}\,(^p\sC)$. By the definition of $e$, 
we have 
$$\Hom_{\hat{^p\sC}}(\iota(c_{\lambda}), e) = \bC\, ,\, \Ext_{\hat{^p\sC}}^1(
\iota(c_{\lambda}), e) = 0$$
for each $c_{\lambda}$. Since $c \in {\rm Ob}\, (^p\sC)$ is obtained 
as extenstions of objects in $\{c_{\lambda}\}_{\lambda \in \Lambda}$ 
in finitely many times, say $n(s)$, and $\iota$ is an exact functor, 
it follows that 
$$\dim_{\bC}\, \Hom_{\hat{^p\sC}}(\iota(c), e) = n(c) < \infty\, .$$
This implies the assertion (1). 

Let us show the assertion (2). Assuming to the contrary that there is 
$p \in {\rm Ob}\, (\sD)$ such that $E(-) \simeq \Hom_{\sD}(-, p)$ as functors, 
we shall derive a contradiction. 

\begin{claim}\label{abelian}
We can take $p$ in ${\rm Ob}\, (^p\sC)$. 
\end{claim}

\begin{proof}
Since $e$ is an injective object of $\hat{^p\sC}$, we have $E(c[n]) = 0$ 
for any $c \in {\rm Ob}\, (^p\sC)$ and for any $n \in \bZ \setminus \{0\}$. 
Since $p$ is represented by a bounded complex of objects in 
${\rm Ob}\,(^p\sC)$, it follows that the cohomology objects $H_{^p\sC}^n(p)$ 
are $0$ for all $n \in \bZ \setminus \{0\}$. Thus $p \simeq p'$ for some 
$p' \in {\rm Ob}\,(^p\sC)$. We may replace $p$ by $p'$.  
\end{proof} 

By Claim (\ref{abelian}), there are $f \in {\rm Ob}\, (\sF)$, $t \in {\rm Ob}\, (\sT)$ and an exact sequence in $^p\sC$:
$$0 \lra f \lra p \lra t[-1] \lra 0\, .$$
As $X$ is very simple, the union $S$ of the support of $t$ and the locus where 
$f$ is not locally free consists of finitely many points, say, 
$S = \{x_1, x_2, \ldots , x_m\}$. 
Let $x \in X \setminus S$. 
Since $\dim\, X \ge 1$, such a point $x$ certainly exists. Since 
$\bC_x[-1]$ is a simple object of $^p\sC$, by definition of $e$, we have 
$E(\bC_x[-1]) \simeq \bC$. Thus, the next claim would give a 
contradiction {\it as desired}.

\begin{claim}\label{setting}
$E(\bC_x[-1]) = 0$.  
\end{claim} 

\begin{proof} From the exact sequence 
$0 \lra f \lra p \lra t[-1] \lra 0$ 
in $^p\sC$, we obtain an exact sequence
$$\Hom_{^p\sC}(\bC_x[-1], f) \lra \Hom_{^p\sC}(\bC_x[-1], p) \lra \Hom_{^p\sC}(\bC_x[-1], t[-1])\,.$$
This is the same as the exact sequence $(*)$:
$$\Hom_{\sD}(\bC_x[-1], f) \lra \Hom_{\sD}(\bC_x[-1], p) \lra \Hom_{\sD}(\bC_x[-1], t[-1])\,.$$
Since $t, \bC_x \in {\rm Ob}\, (\sT) \subset {\rm Ob}\, (\sD)$ 
and $f \in {\rm Ob}\, (\sF) \subset {\rm Ob}\, (\sD)$, it follows that
$$\Hom_{\sD}(\bC_x[-1], t[-1]) \simeq \Hom_{\sD}(\bC_x, t) \simeq \Hom_{\sC}(\bC_x, t) = 0\,.$$
Here the last equality follows from $x \not\in\, {\rm Supp}\,(t)$. 
We also have 
$$\Hom_{\sD}(\bC_x[-1], f) \simeq \Hom_{\sD}(\bC_x, f[1]) \simeq 
\Ext_{\sC}^1(\bC_x, f)\, .$$ 
We compute the last term. The computaion is now in the {\it usual} 
category ${\rm Coh}\, X$. For a sheaf $s$ and a point $x$ on $X$, as usual, 
we denote by $s_x$ the germ of $s$ at $x$. As $f$ is torsion free, we have 
${\mathcal Hom}(\bC_x, f) = 0$. Note 
that $(\bC_x)_y = 0$ unless $y = x$. Recall also that $\mathcal O_{X, x}$ is a noetherian regular local ring of dimension $\ge 2$ and $f$ is locally free 
at $x$ by $x \not\in S$.  Thus,  
$$({\mathcal Ext}^1(\bC_x, f))_y \simeq {\mathcal Ext}^1((\bC_x)_y, f_y)  = 0$$ for all $y \in X$ including the case $y=x$. Thus, ${\mathcal Ext}^1(\bC_x, f) = 0$. This is only the place where we used the fact that $\dim_{\bC} X \ge 2$. 
Now, by the local-global spectral sequence, we obtain 
$\Ext_{\sC}^1(\bC_x, f) = 0$. 
Hence, by the exact sequence $(*)$, $E(\bC_x[-1]) = \Hom_{\sD}(\bC_x[-1], p) = 0$. 
\end{proof}
This contradiction completes the proof of (2). 
\end{proof}

\vskip .2cm \noindent Keiji Oguiso \\ 
Department of Mathematics\\
Osaka University\\ 
Toyonaka 560-0043 Osaka, Japan\\
oguiso@math.sci.osaka-u.ac.jp


\vskip 1cm

\begin{thebibliography}{999999}

\bibitem[BBD83]{BBD83} A. Beilinson, J. Bernstein, P. Deligne, \textit{Faisceaux Pervers}, Ast\'erisque {\bf 100} Soc. Math de France, Paris (1983).
 
\bibitem[BHPV04]{BHPV04} Barth, W. P., Hulek, K., Peters, C. A. M.,
Van de Ven,A.: \textit{Compact complex surfaces},
Springer(2004).

\bibitem[BV03]{BV03} Bondal, A., Van den Bergh, M., \textit{Generators and representability of functors in commutative and noncommutative geometry}, Mosc. Math. J. {\bf 3} (2003) 1--36.

\bibitem[Fu83]{Fu83} Fujiki, A., \textit{On the structure of compact complex manifolds in ${\mathcal C}$. Algebraic varieties and analytic varieties (Tokyo, 1981)}, Adv. Stud. Pure Math. {\bf 1} North-Holland, Amsterdam (1983) 231--302. 

\bibitem[Ga62]{Ga62} Gabriel, P., \textit{Des cat\'egories ab\'eliennes}, 
Bull. Soc. Math. France {\bf 90} (1962) 323--448.

\bibitem[GR84]{GR84} Grauert, H., Remmert, R., \textit{Coherent analytic 
sheaves}, Grundlehren der Mathematischen Wissenschaften {\bf 265} Springer-Verlag, Berlin, 1984. xviii+249 pp. 

\bibitem[HRS96]{HRS96} Happel, D., Reiten, I., Smalo, S., \textit{Tilting in abelian categories and quasitilted algebras}, Mem. Amer. Math. Soc. 
{\bf 120} (1996). 

\bibitem[KS90]{KS90} Kashiwara, M., Schapira, P., \textit{Sheaves on manifolds. With a chapter in French by Christian Houzel}, Grundlehren der Mathematischen Wissenschaften {\bf 292} Springer-Verlag, Berlin(1994) x+512 pp. 

\bibitem[Me07]{Me07} Meinhardt, S., \textit{Stability conditions on generic complex tori}, preprint, arXiv:0708.3053v2. 

\bibitem[Ne90]{Ne90} Neeman, A., \textit{The derived category of an exact category}, J. Algebra {\bf 135} (1990) 388--394. 

\bibitem[Ne96]{Ne96} Neeman, A,, \textit{The Grothendieck duality theorem via Bousfield's techniques and Brown representability}, J. Amer. Math. Soc. 
{\bf 9} (1996) 205--236. 

\bibitem[RV02]{RV02} 
Reiten, I., Van den Bergh, M., \textit{Noetherian hereditary abelian categories satisfying Serre duality}, J. Amer. Math. Soc. {\bf 15} (2002) 295--366.

\bibitem[Ro08]{Ro08} Rouquier, R. \textit{Dimensions of triangulated categories}, J. K-Theory {\bf 1} (2008) 193--256.

\bibitem[Sch82]{Sch82} Schuster, H.-W., \textit{Locally free resolutions of coherent sheaves on surfaces}, J. Reine Angew. Math. {\bf 337} (1982) 159--165. 

\bibitem[TV08]{TV08} To\"en, B., Vaqui\'e, M., \textit{Alg\'ebrisation des vari\'et\'es analytiques complexes et cat\'egories d\'eriv\'ees} Math. Ann. 
{\bf 342} (2008) 789--831. 

\bibitem[Ve04]{Ve04} Verbitsky, M., \textit{Coherent sheaves on generic compact tori}, Algebraic structures and moduli spaces, 229--247, CRM Proc. Lecture Notes, {\bf 38} Amer. Math. Soc., Providence, RI, 2004. 

\bibitem[Ve08]{Ve08} Verbitsky, M., \textit{Coherent sheaves on general $K3$ surfaces and tori}, Pure Appl. Math. Q. {\bf 4} (2008) 651--714. 

\end{thebibliography}
\end{document}